\documentclass[12pt,oneside]{amsart}
\usepackage[utf8]{inputenc}
\usepackage[english]{babel}
\usepackage{amsfonts}
\usepackage{amssymb, amsmath}
\usepackage{amsthm}
\usepackage{graphicx}
\usepackage{comment}
\usepackage{pgf}
\usepackage{graphics,epsfig, subfigure}
\usepackage{url}
\usepackage{srcltx}
\usepackage{hyperref}

\theoremstyle{plain}
      \newtheorem{theorem}{Theorem}
       \newtheorem{proposition}{Proposition}
      \newtheorem{lemma}{Lemma}
      
      \newtheorem{problem}{Problem}
      \newtheorem*{problem*}{Problem}
      \newtheorem{remark}{Remark}
      \newtheorem*{remark*}{Remark}
      \theoremstyle{definition}

      \title
[On recognizing  shapes of polytopes from their shadows]
{On recognizing  shapes of polytopes from their shadows}
\author{Sergii Myroshnychenko}
\begin{document}

\begin{abstract}
Let $P$ and $Q$ be two convex polytopes both contained in the interior of an Euclidean ball $r\textbf{B}^{d}$.
We prove that $P=Q$ provided that their sight cones from any point on the sphere $rS^{d-1}$ are congruent.  We also prove an analogous result for spherical projections.
\end{abstract}

\maketitle

\section{Introduction}

An orthogonal projection of a set on a subspace can be thought of as a shadow of this set, with the source of light located infinitely far away.

There are many problems regarding orthogonal projections. For instance, one of the long-standing problems of convex geometry (\cite{Ga}, Problem 3.2, p. 125)

\begin{problem}\label{pr2}
Let $2\le k\le d-1$. Assume that $K$ and $L$ are convex bodies in ${\mathbb E}^d$ such that the projections $K|H$ and $L|H$ are congruent for all subspaces $H \subset \mathbb{E}^d$, $\dim H = k$. Is $K$ a translate of $\pm L$?
\end{problem}

Here we say that two subsets  $A, B \subset \mathbb{E}^d$ are congruent, if there exists an orthogonal transformation $\varphi \in O(d)$ and a vector $b \in \mathbb{E}^d$, such that $\varphi(A) + b = B$.

There is also a broad class of tomography type problems related to \emph{isoptic} characterization of convex bodies. For example, the following
\begin{problem}
Let $C \subset \mathbb{E}^2$ be a smooth convex curve that is contained in the interior of a certain circle $S$. Assume that from any point on $S$ curve $C$ subtends the same angle. Can we conclude that $C$ is a circle?
\end{problem}
It was shown by Green (see \cite{Gr}) that this problem has a negative answer. However, Klamkin (see \cite{Kl}) conjectured that if there exist two distinct circles with the same property for the curve $C$, then the answer is affirmative; this was proved in \cite{N}. Similar result was generalized to higher dimensions (the case of a convex body contained in the interior of a sphere) in \cite{KO}.

One can also ask an analogous question for non-constant angles. It was shown in \cite{KK} that in the class of convex polygons the measure of the subtended angles as a function on the circle $S$ defines the set uniquely. At last, a very interesting result from \cite{Mat} shows that not necessarily the measure, but the shape of a subtended solid angle from points on a sphere containing a convex body characterizes balls uniquely.

Following the above considerations, we ask the following

\begin{problem}\label{spr1}
Let $K, L$ be two convex bodies contained in the interior of a ball $r\textbf{B}^{d}$, $d\ge 3, r>0$. Assume that for any point $z$ (a source of light) on the sphere $rS^{d-1}$, the spherical projections $K_z$ and $L_z$ of the bodies are congruent. Does it follow that $K=L$? (see Figure \ref{pic}).

\end{problem}

\begin{figure}[h]
  \centering
  \includegraphics[scale=0.3]{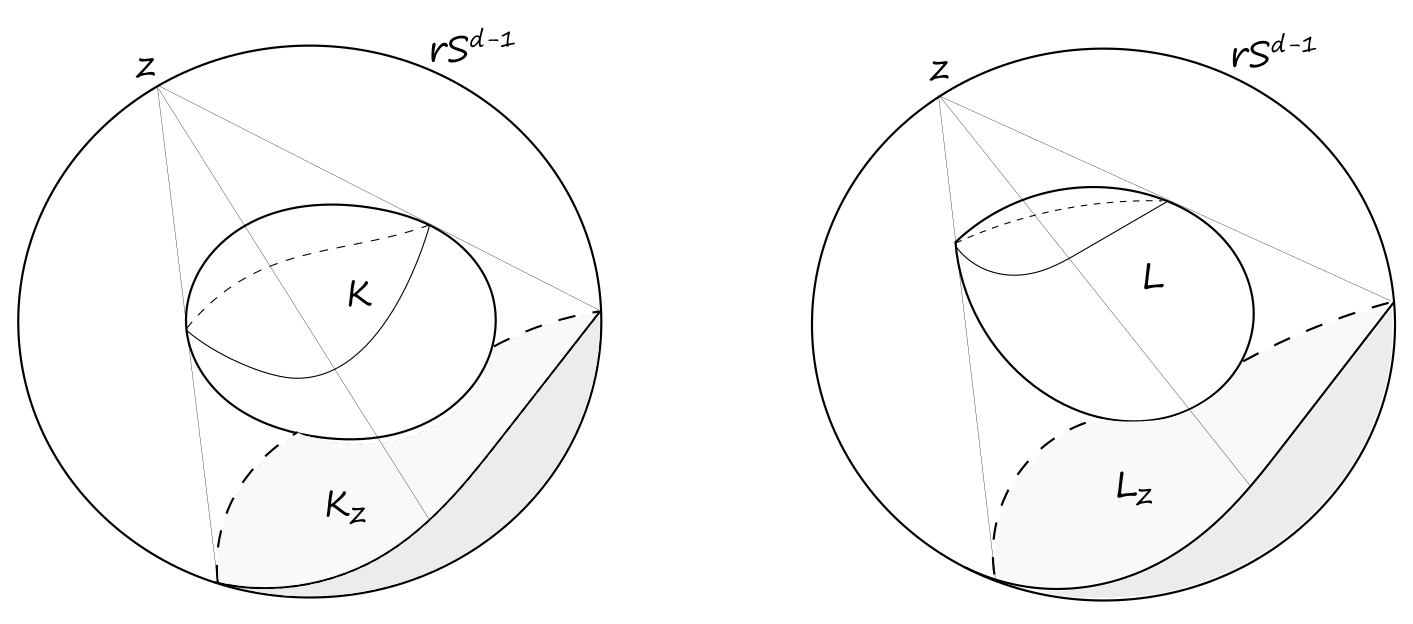}\\
  \caption{Congruency of spherical projections:\newline
   $\forall z \in S^{d-1}  \quad \exists \varphi_z \in O(d), \quad \varphi_z(K_z) = L_z.$}
\label{pic}
\end{figure}

Here, $K_z = \{x \in rS^{d-1}: xz \cap K\neq\emptyset \}$, and $xz$ denotes the interval with end points $x$ and $z$. Also, in this case we say that two subsets  $A, B \subset rS^{d-1}$ are congruent, if there exists an orthogonal transformation $\varphi \in O(d)$, such that $\varphi(A) = B$.

It's not difficult to see that in the case $d=2$ Problem \ref{spr1} has a negative answer, which follows from the result of Green (see \cite{Gr}). This result is analogous to the well-known fact that in $\mathbb{E}^2$ there exist many not congruent bodies of constant width (see \cite{Ga}, p. 131). However, it is known that for $d=2$ in the class of convex polygons Problem \ref{spr1} has the affirmative answer (see \cite{KK}).

We observe that when the radius $r$ of the ball containing the bodies approaches infinity, the above problem becomes very similar to Problem \ref{pr2}. If one attempts to solve Problem \ref{spr1}, one sees that one of the main difficulties is related to the lack of convenient notions (for example, a support function as in the case of orthogonal projections).
Not a difficult consideration shows the affirmative result in the class of Euclidean balls for Problem \ref{spr1}.

\begin{proposition}\label{balls}
	Let $d\geq 2, r> 0$ and let $K, L$ be two Euclidean balls contained in the interior of a ball $r\textbf{B}^{d}$. Assume that for any point $z$ on the sphere $rS^{d-1} = \partial(r\textbf{B}^d)$, the spherical projections $K_z$ and $L_z$ of the balls are congruent. Then  $K=L$.
\end{proposition}

It is also worth stating another similar natural problem from \cite{KK}
\begin{problem} \label{cones}
	If $K$ and $L$ are two convex bodies inside the sphere $rS^{d-1}$, and for each point of $rS^{d-1}$ the supporting cones of $K$ and $L$ from this point are congruent, then is it true that $K=L$?
\end{problem}
It follows from the above mentioned result of Matsuura \cite{Mat} that the answer is affirmative if one of the bodies is a ball. Also, Bianchi and Gruber \cite{BG} proved that if one of the bodies is an ellipsoid then the other body must also be an ellipsoid. In general, the problem remains open.

The main results of this paper show that both Problems \ref{spr1} and \ref{cones} have the affirmative answers  in the class of convex polytopes. 

\begin{theorem}\label{sth1}
	Let $d\geq 3, r> 0$ and let $P,Q$ be two convex polytopes contained in the interior of a ball $r\textbf{B}^{d}$. Assume that for any point $z$ on the sphere $rS^{d-1} = \partial(r\textbf{B}^d)$, the support cones $C(z,P)$ and $C(z,Q)$ of the polytopes are congruent. Then $P=Q$.
\end{theorem}

\begin{theorem}\label{sth2}
	Let $d\geq 3, r> 0$ and let $P,Q$ be two convex polytopes contained in the interior of a ball $r\textbf{B}^{d}$. Assume that for any point $z$ on the sphere $rS^{d-1} = \partial(r\textbf{B}^d)$, the spherical projections $P_z$ and $Q_z$ of the polytopes  are congruent. Then $P=Q$.
\end{theorem}

The affirmative answer to Problem \ref{pr2} in the class of convex polytopes was obtained in \cite{MR}; it is not difficult in the class of Euclidean balls. Quite surprisingly shifts play no role in these settings.

\vspace{0.3cm}
\textbf{Acknowledgments:} The author is grateful to Vlad Yaskin and Dmitry Ryabogin for many fruitful and interesting discussions.

\section{Notation}

Euclidean space of dimension $d$ is denoted by $\mathbb{E}^d$. We use the notation $r\textbf{B}^d=\{x\in \mathbb{E}^d: |x| \le r\}$
for the Euclidean $d$-dimensional ball of radius $r$ in $\mathbb{E}^d$, and
$rS^{d-1}=\{x\in \mathbb{E}^d: |x|=r\}$
for the $(d-1)$-dimensional sphere of radius $r$, which is the boundary of the ball $r\textbf{B}^d$. The orthogonal group of dimension $d$ is denoted by $O(d)$. For any two points $x, y \in \mathbb{E}^d$, the closed interval connecting them is denoted by $xy = \{tx+(1-t)y: t \in [0,1]\}$. For any two points $p,q \in rS^{d-1}$, the shortest arc of a great circle connecting them is denoted by $[pq]$.
The interior of a set $A$ (denoted by $int A$) is the set of points $x\in A$ such that there exists an open set containing $x$ which is fully contained in $A.$ The boundary of a set $A$ (denoted by $\partial A$) is the set defined by $\partial A=A\setminus int A.$

By the shadow boundary we mean the pre-image of the boundary of the shadow. More precisely, for any body $M \subset int (r\textbf{B}^d)$ the shadow boundary from a point $z \in rS^{d-1}$ is defined as
$$
\partial_zM =\{zx \cap M: x \in \partial M_z\},
$$
where $\partial M_z$ is the relative boundary of $M_z$ on the sphere.

For any finite set of points $S=\{x_i\}_{i=1}^N \subset \mathbb{E}^d$, the convex hull of $S$ is defined as $Conv(S) = \left\{ \sum_{i=1}^N \lambda_i x_i: \sum_{i=1}^N \lambda_i=1; \lambda_i \geq 0 \quad \forall i\right\}$. A convex polytope $P\subset {\mathbb E^d}$ is a convex body that is the convex hull of finitely many points. The extreme points of the convex hull are called vertices. 

\textit{A supporting (sight) cone} from a point $z$ outside of a body $K$ is defined as
$$
C(z,K) = \cup_{x \in K} \{z + s(x-z):  \quad s \geq 0\}.
$$

In particular, the spherical projection then can be found as 
$$K_z = \big( C(z, K) \cap rS^{d-1} \big) \setminus \{z\}.$$

Notice that if $P$ is a polytope, $z \not \in P$, and $\{v_i\}_i$ is the set of its vertices in the shadow boundary $\partial_zP$, then for any $y \in \partial C(z,P)$, we have
$$
y= z + \sum_{v_j \in \partial_zP} \lambda_j (v_j - z) \quad \textrm{for some} \quad \lambda_j \geq 0, \quad \{v_j\}_j \subseteq \{v_i\}_i.
$$

In this case, the cone is polyhedral and its \textit{spanning edge} $l_{zv_i}$ is a ray defined as
$$
l_{zv_i} = z + t (v_i - z), \quad t \geq 0.
$$

\section{Proof of Proposition \ref{balls}}

\begin{proof}
	Assume that the balls $K,L$ do not coincide. Denote by $O_K$ and $O_L$ their centers, and by $r_K$ and $r_L$ the radii of $K$ and $L$ respectively. Let $l$ be the line passing through both of the centers. Denote by $z_1$ and $z_2$ the points of the intersection of $l$ with $rS^{d-1}$ (see Figure \ref{spherical}).
	
	\begin{figure}[h]
		\centering
		\includegraphics[scale=0.3]{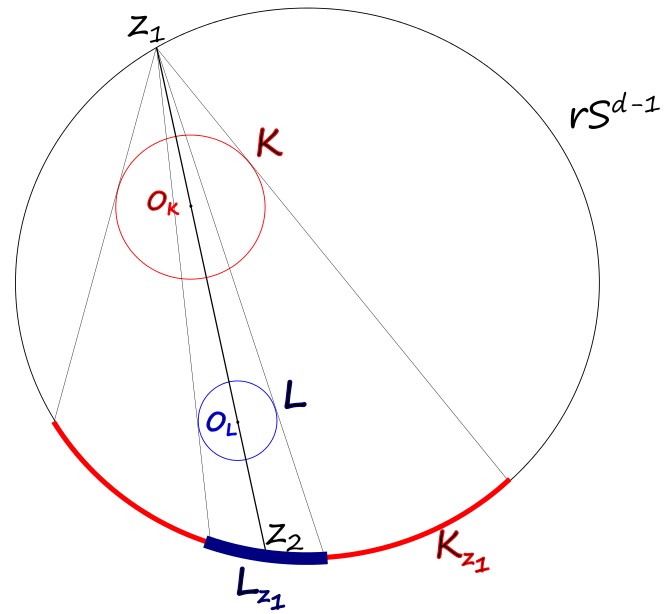}\\
		\caption{Spherical projections of Euclidean balls.}\label{spherical}
	\end{figure}
	
	If $r_L < r_K$, then $L_{z_1} \subsetneq K_{z_1}$. Hence, there does not exist $\varphi_{z_1} \in O(d)$, such that $\varphi_{z_1}(K_{z_1})=L_{z_1}$, which implies that we must have $r_L \geq r_K$. Now, if $r_K < r_L$, we repeat the same argument for $z_2$. Thus, $r_K = r_L$.

	Now, if $|z_1O_K| < |z_1O_L|$, then $C(z_1, L) \subsetneq C(z_1,K)$, which would imply that $L_z \subsetneq K_z$. Similarly, it can be shown that the case $|z_1O_L| < |z_1O_K|$ is not possible either, which implies that $O_K = O_L$. We conclude that $K = L$.
	
\end{proof}

\section{Proof of Theorem \ref{sth1}}

\subsection{The main idea}

The main idea of the proof is to show that the sets of edges of both polytopes coincide. More precisely,  let $\Pi$ be the set of all affine subspaces containing a facet of polytopes $P$ or $Q$. Then $\Pi \cap rS^{d-1}$ is a finite union of sub-spheres of non-trivial radii. Let $U_1$ be a connected component of $U = rS^{d-1} \setminus (\Pi \cap rS^{d-1})$. For any two points $z_1$ and $z_2$ in $U_1$, we have $\partial_{z_1}P =\partial_{z_2} P$ and $\partial_{z_1}Q =\partial_{z_2} Q$, which follows from the considerations in \cite{KP}.

After this, we will prove that the edges of the shadow boundaries $\partial_z P$ and $\partial_z Q$, $z \in U$, are in bijective correspondence. Moreover, we will show that they coincide. This would imply that $P = Q$.

\subsection{Preliminaries}

 Congruency of supporting cones $C(z,P)$ and $C(z,Q)$ implies that both cones have the same number of spanning edges which are congruent. In other words, there exists a fixed $\varphi_z \in O(d)$ such that
\begin{equation}\label{map}
 \varphi_z\Big(\frac{v_i - z}{|v_i-z|}\Big) =  \frac{\tilde{v}_j - z}{|\tilde{v}_j-z|},
\end{equation}
 where $v_i$ and $\tilde{v}_j$ are vertices in $\partial_zP$ and $\partial_zQ$, respectively, that define congruent spanning edges $l_{zv_i}$ and $l_{z\tilde{v}_j}$ of the supporting cones (see Figure \ref{suppcones}).
 
 \begin{figure}[h]
 	\centering
 	\includegraphics[scale=0.4]{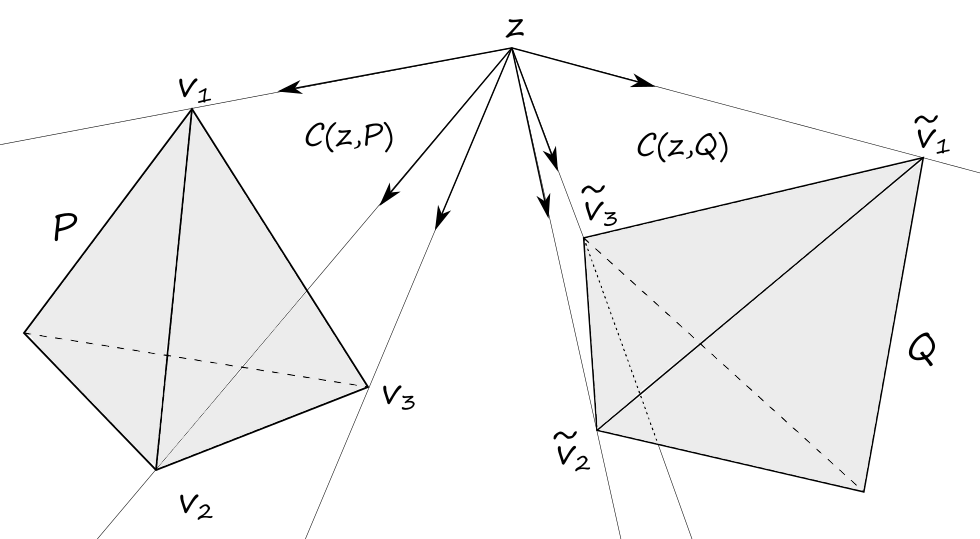}
 	\caption{Spanning edges are pairwise congruent, $\varphi_z(l_{zv_i}) = l_{z\tilde{v}_j}$. }
 	\label{suppcones}
 \end{figure}

This implies that any ``light source'' $z \in rS^{d-1}$ corresponds to a permutation $\sigma_z$ of a finite set of spanning edges of both cones $C(z,P)$ and $C(z,Q)$. According to the considerations from \cite{KP} (see the proof of Theorem 1 there) and the condition of congruency of the cones, for each connected component $U_i$, these sets have the same number of elements, say, $k$. Hence, for any fixed $z\in U_1$ we obtain a bijective correspondence $f_z$ induced by $\varphi_z$ between the set of all vertices  $\{v_1,v_2,...,v_k\} $ of the shadow boundary $\partial_z P$ and the set of all  vertices $\{\tilde{v}_1,\tilde{v}_2,...,\tilde{v}_k\}$ of the shadow boundary $\partial_z Q$.


Take a closed spherical cap with a  non-empty interior $W \subset U_1$. For any $z \in W$, we have at least one map $f_z: \{v_1,v_2,...,v_k\} \to \{\tilde{v}_1,\tilde{v}_2,...,\tilde{v}_k\} $, such that $f_z(v_i) = \tilde{v}_{\sigma_z(i)}$,  and $\sigma_z$ is a permutation of the set $\{1,2,...,k\}$ satisfying $  \varphi_z\big(\frac{v_i - z}{|v_i-z|}\big) =  \frac{\tilde{v}_{\sigma(i)} - z}{|\tilde{v}_{\sigma(i)}-z|}$. The set  of all such possible maps $\{f_z\}_{z \in W}$ is finite. We have
$$
W = \bigcup\limits_{\sigma \in {\mathcal P}_k} V_{\sigma},\quad V_{\sigma} = \{z \in W:\, \exists f_z \,\,\textrm{such that}\,\,f_z(v_i)=\tilde{v}_{\sigma(i)}\quad\forall i=1,\dots,k\},
$$
where $ {\mathcal P}_k$ is the set of all permutations of $\{1,2,...,k\}$.

Observe that each $V_{\sigma}$ is a closed set (it might be empty). Indeed, let $\{z_m\}_{m=1}^{\infty}$ be a convergent sequence of points of a non-empty $V_{\sigma}$, and let $\lim\limits_{m\to\infty}z_m=z$.
We have $\varphi_{z_m}\big(\frac{v_i - z_m}{|v_i - z_m|}\big)=\frac{\tilde{v}_{\sigma(i)}- z_m}{|\tilde{v}_{\sigma(i)} - z_m|}$, where $\sigma$ is independent of $z_m$, since  $ \forall m \in \mathbb{N}, z_m \in V_{\sigma}$. Since $O(d)$ is compact, there exists $\tilde{\varphi}_{z} = \lim\limits_{m \to \infty} \varphi_{z_m}$.  Note that $\tilde{\varphi}_z$ may not coincide with $\varphi_z$. This yields
$$
\tilde{\varphi}_{z}\Big(\frac{v_i - z}{|v_i - z|}\Big)=\frac{\tilde{v}_{\sigma(i)}- z}{|\tilde{v}_{\sigma(i)} - z|}
$$
In other words, there exists $\tilde{\varphi}_{z}$, such that the corresponding $\tilde{f}_{z}$ satisfies $\tilde{f}_{z}(v_i)=\tilde{v}_{\sigma(i)}, \forall i=1,\dots,k$. This means that $z \in V_{\sigma}$ and $V_{\sigma}$ is a closed set.

By the Baire category Theorem (see, for example,  \cite{R}, pages 42-43) there exists a permutation $\sigma_o$ such that the interior $U_o$ of $V_{\sigma_o}$ is non-empty.

\subsection{Two-dimensional faces of supporting cones}

For each edge of $P$ with vertices $x,y$, there is a $U_0 \subset U_i$ as above and an edge of $Q$ with vertices $p,q$ such that under the permutation $\sigma_0$ we have $\sigma_0(x) = p, \sigma_0(y) = q$.


\begin{figure}[h]
	\centering
	\includegraphics[scale=0.3]{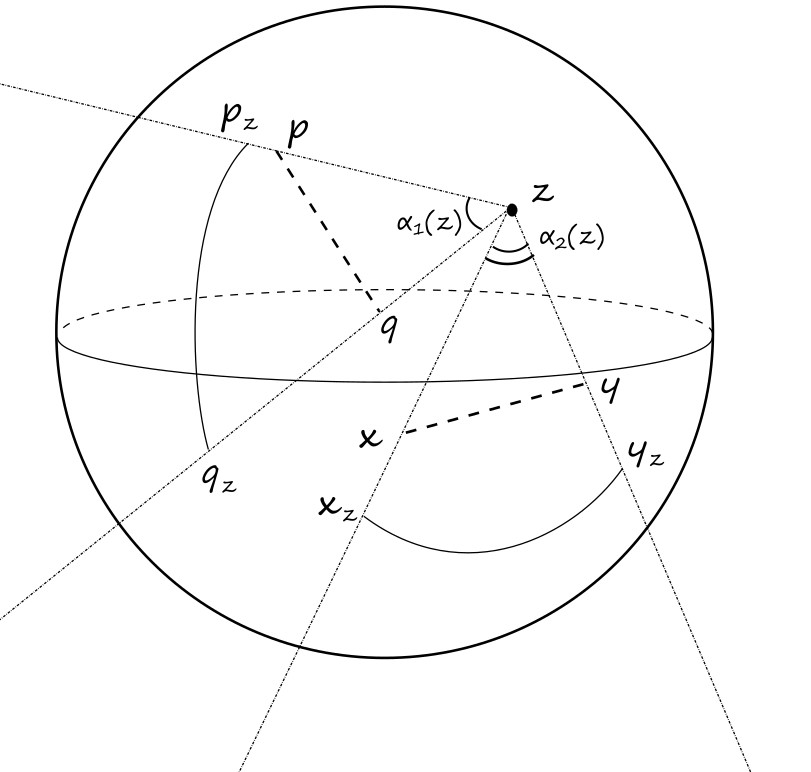}\\
	\caption{$\angle xzy = \angle pzq$ for any $z \in U_0$.}\label{ed}
\end{figure}

 From the condition of congruency of the supporting cones, the angles are equal $\alpha_2(z)=\angle pzq = \angle xzy = \alpha_1(z)$ for any $z \in U_0$ (see Figure \ref{ed}). We may extend both of the functions $\alpha_1(z), \alpha_2(z)$ naturally on the whole sphere, i.e. $\alpha_1(z) = \angle xzy$, $\alpha_2(z) = \angle pzq$ for any $z \in r S^{d-1}$.

Recall that $rS^{d-1}$ is an analytic image of a $(d-1)$-dimensional parallelepiped $\Lambda = [0, \pi]\times[0,\pi]\times\ldots[0,\pi] \times [0,2\pi] \subset \mathbb{E}^{d-1}$ with the parametrization $z = (x_1,\ldots,x_{d-1},x_d)$, where
$$
x_1 = r \cos \varphi_1
$$
$$
x_2 = r \sin \varphi_1 \cos \varphi_2
$$
$$
x_3 = r \sin \varphi_1 \sin \varphi_2 \cos \varphi_3
$$
$$
\ldots
$$
$$
x_{d-1} = r \sin \varphi_1 \ldots \sin \varphi_{d-2} \cos \varphi_{d-1}
$$
$$
x_d = r \sin \varphi_1 \ldots \sin \varphi_{d-2} \sin \varphi_{d-1},
$$
and $\{\varphi_1,\ldots,\varphi_{d-2}\} \in [0,\pi]$, $\varphi_{d-1} \in [0,2\pi]$.
Also,
$$
\cos \alpha_1(z) = \frac{(z-x) \cdot (z-y)}{|z-x||z-y|}, \quad \cos \alpha_2(z) = \frac{(z-p) \cdot (z-q)}{|z-p||z-q|}.
$$
We see that $\cos \alpha_i(z)$ is an analytic function on $\Lambda$, and $$
\cos \alpha_1(z(x_1, x_2,\ldots,x_{d-1})) = \cos \alpha_2(z(x_1, x_2,\ldots,x_{d-1})),
$$
for any $ (x_1, x_2, \ldots, x_{d-1}) \in z^{-1}(U_0) \subset \Lambda\subset \mathbb{E}^{d-1}$, where $z^{-1}(U_0)$ stands for the pre-image of $U_0$ in $\Lambda$. Hence, by \cite{O}, the functions coincide on the whole $\Lambda$. Observe that, by the construction, $\alpha_i(z) \in [0,\pi)$, hence $\cos \alpha_1(z) = \cos \alpha_2(z)$ is equivalent to $\alpha_1(z) = \alpha_2(z)$. Thus,
\begin{equation}\label{angles}
\alpha_1(z) \equiv \alpha_2(z) \quad \forall z \in rS^{d-1}.
\end{equation}
Consider a point $z_0$, such that $z_0 \in rS^{d-1} \cap l$, where $l$ is the line containing the segment $pq$. Notice that $\alpha_1(z_0)=0$, hence $\alpha_2(z_0)=0$. This implies that both segments belong to the same line $l$. We consider the restriction of $\alpha_i(z)$ onto $span\{l, O\} \cap rS^{d-1} \cong rS^1$. If $l$ passes through the origin, we take any two dimensional subspace containing it. We have
$$
\alpha_1(z) \equiv \alpha_2(z) \quad \forall z \in rS^1.
$$
By Lemma 2.1 from \cite{KK}, we conclude that $x=p, y = q.$

Now, since all the corresponding vertices coincide in each connected component $U_i$ for both polytopes, then $P$ and $Q$ coincide as well.

\begin{remark}
	In the proof above, the sphere $rS^{d-1}$ that contains both polytopes $P$ and $Q$ can be substituted with any convex closed analytic hypersurface, since we did not use the parametrization of the sphere, but only its analyticity.
\end{remark}

\section{Proof of Theorem \ref{sth2}}


To prove Theorem \ref{sth2}, it is enough to repeat the above proof of Theorem \ref{sth1} with some small modifications. First, congruency of $P_z$ and $Q_z$ implies that their vertices are congruent. Here by a vertex of a spherical projection we mean the image $(v_i)_z$ of  a vertex in the shadow boundary $v_i \in \partial_zP$. Then, by the analogy to \eqref{map}, there exists $\varphi_z \in O(d)$, such that $\varphi_z\big((v_i)_z\big) = (\tilde{v}_j)_z$, where $v_i$ and $\tilde{v}_j$ are vertices in $\partial_zP$ and $\partial_zQ$ respectively. This defines permutations between the sets of vertices. We proceed as in the previous proof, until the consideration of the angles $\angle xzy$ and $\angle pzq$. Let $\Pi(x,y,z)$ be the $2$-dimensional plane containing three distinct points $x,y,z$. The projections of edges $xy$ and $pq$ are the edges of spherical projections, i.e. arcs of (small) circles $(xy)_z$ and $(pq)_z$. Under the orthogonal transformation, they must coincide $\varphi_z\big((xy)_z\big) = (pq)_z$. This implies that circles $\Pi(x,y,z) \cap rS^{d-1}$ and $\Pi(p,q,z) \cap rS^{d-1}$ have equal radii. The equality of lengths of the arcs implies equality of the angles $\angle xzy =\angle pzq$ as before in \eqref{angles}. The conclusion follows.

\begin{remark}
It seems likely that congruency of sight cones and congruency of spherical projections from any point on the sphere are equivalent. However, we remark that for a fixed point this is not the case. Consider the sphere $S$ defined as $x^2 + y^2 + (z+1)^2=1$ and the cone $C_1$, $z^2 = x^2 + y^2, z \leq 0$. Then the intersection of these two surfaces is a unit circle $S_1$, $x^2+y^2=1, z= -1$. Now consider a unit circle $S_2$ obtained by intersecting sphere $S$ with the plane $z=x-1$. Both $S_1$ and $S_2$ are congruent, but the cone $C_2 = span\{0, S_2 \}$ is strictly elliptical (not circular), thus, not congruent to $C_1$.
\end{remark}

\end{document}